\documentclass[reqno]{amsart}

\usepackage{amsmath,amsthm,amssymb,amscd}
\newtheorem{thm}{Theorem}[section]

\newtheorem{lem}[thm]{Lemma}

\theoremstyle{definition}

\newtheorem{ques}[thm]{Question}
\theoremstyle{remark}
\newtheorem{rem}[thm]{Remark}
\newtheorem*{ex}{Example}
\numberwithin{equation}{section}

\begin{document}
	
	%
	%
	%
	%
	%
	%
	%
	%

	\title[$p$-Adic quotient sets: linear recurrence sequences]
	{$p$-Adic quotient sets: linear recurrence sequences with reducible characteristic polynomials}

	\author[Deepa Antony]{Deepa Antony (ORCID: 0000-0002-4214-2889)}
	
	\address{
		Department of Mathematics \\
		Indian Institute of Technology Guwahati \\
		Assam, India, PIN- 781039}
	
	\email{deepa172123009@iitg.ac.in}
	
	\thanks{Journal: To appear at Canadian Mathematical Bulletin}
	\author[Rupam Barman]{Rupam Barman (ORCID: 0000-0002-4480-1788)}
	\address{Department of Mathematics \\
		Indian Institute of Technology Guwahati \\
		Assam, India, PIN- 781039}
	\email{rupam@iitg.ac.in}
	\subjclass{Primary 11B05, 11B37, 11E95}
	
	\keywords{$p$-adic number, Quotient set, Ratio set, Linear recurrence sequence}
	
	\date{\today}
	
	\begin{abstract}
		Let $(x_n)_{n\geq0}$ be a linear recurrence sequence of order $k\geq2$ satisfying
		$$x_n=a_1x_{n-1}+a_2x_{n-2}+\dots+a_kx_{n-k}$$ for all integers $n\geq k$, where $a_1,\dots,a_k,x_0,\dots, x_{k-1}\in \mathbb{Z},$ with $a_k\neq0$. 
		In 2017, Sanna posed an open question to classify primes $p$ for which the quotient set of $(x_n)_{n\geq0}$ is dense in $\mathbb{Q}_p$. In a recent paper, we showed that if the characteristic polynomial of the recurrence sequence has a root $\pm \alpha$, where $\alpha$ is a Pisot number and if $p$ is a prime such that the characteristic polynomial of the recurrence sequence is irreducible in $\mathbb{Q}_p$, 
		then the quotient set of $(x_n)_{n\geq 0}$ is dense in $\mathbb{Q}_p$. In this article, we answer the problem for certain linear recurrence sequences whose characteristic polynomials are reducible over $\mathbb{Q}$. 
	\end{abstract}
	
	\maketitle
	\section{Introduction and statement of results} 
	For a set of integers $A$, the set $R(A)=\{a/b:a,b\in A, b\neq 0\}$ is called the ratio set or quotient set of $A$. Many authors have studied the denseness of ratio sets of different subsets of $\mathbb{N}$ in the positive real numbers. 
	See for example \cite{real-1, real-2, real-3, real-4, real-5, real-6, real-7, real-8, real-9, real-10, real-11, real-12, real-13, real-14}. An analogous study has also been done for algebraic number fields, see for example \cite{algebraic-1, algebraic-2}. 
	\par For a prime $p$, let $\mathbb{Q}_p$  denote the field of $p$-adic numbers. In recent years, the denseness of ratio sets in $\mathbb{Q}_p$ have been studied by several authors, 
	see for example \cite{cubic, diagonal, Donnay, garciaetal, garcia-luca, miska, piotr, miska-sanna, Sanna1}. Let $(F_n)_{n\geq 0}$ be the sequence of Fibonacci numbers, 
	defined by $F_0=0$, $F_1=1$ and $F_n=F_{n-1}+F_{n-2}$ for all integers $n\geq 2$. In \cite{garcia-luca}, Garcia and Luca showed that the ratio set of Fibonacci numbers is dense in $\mathbb{Q}_p$ for all primes $p$. 
	Later, Sanna \cite[Theorem 1.2]{Sanna1} showed that, for any $k\geq 2$ and any prime $p$, the ratio set of the $k$-generalized Fibonacci numbers is dense in $\mathbb{Q}_p$ 
and made the following open question.
	\begin{ques}\cite [Question 1.3]{Sanna1}\label{question-1}
		Let $(S_n)_{n\geq0}$ be a linear recurrence sequence of order $k\geq2$ satisfying
		\begin{align*}
			S_n=a_1S_{n-1}+a_2S_{n-2}+\dots+a_kS_{n-k},
		\end{align*}
		for all integers $n\geq k$, where $a_1,\dots,a_k,S_0,\dots,S_{k-1}\in \mathbb{Z}$, with $a_k\neq0$. For which prime numbers $p$ is the quotient set of $(S_n)_{n\geq0}$ dense in $\mathbb{Q}_p?$
	\end{ques}
	In \cite{garciaetal}, Garcia et al. solved the problem partially for second-order recurrences. 
Later, in \cite{reccurence}, we considered $k$th-order recurrence sequences for which $a_k=1$ and initial values $S_0=\cdots=S_{k-2}=0$, $S_{k-1}=1$. We showed that if the characteristic polynomial of the recurrence sequence has a root $\pm \alpha$, where $\alpha$ is a Pisot number and if $p$ is a prime such that the characteristic polynomial of the recurrence sequence is irreducible in $\mathbb{Q}_p$, 
then the quotient set of $(x_n)_{n\geq 0}$ is dense in $\mathbb{Q}_p$. In this article, our objective is to study the denseness of quotient sets of linear recurrence sequences whose characteristic polynomials are reducible over $\mathbb{Q}$. Also, we extend \cite[Theorem 1.9]{reccurence}, which gives condition for the denseness of ratio sets of  second order linear recurrence sequences $(x_n)_{n\geq 0}$ whose characteristic polynomials are of the form $(x-a)^2$, to $k$th order linear recurrence sequences with characteristic polynomials of the form $(x-a)^k$ in the case when the initial values are given as $x_0=x_1=\dots=x_{k-2}=0,x_{k-1}=1$. 
\par In our first theorem, we consider recurrence sequences having characteristic polynomials whose roots are all distinct. 
\begin{thm}\label{new-thm1}
	Let $(x_n)_{n\geq0}$ be a linear recurrence of order $k\geq2$ satisfying
	\begin{align*}
	x_n=b_1x_{n-1}+b_2x_{n-2}+\dots+b_kx_{n-k},
	\end{align*}
	for all integers $n\geq k$, where $b_1,\dots,b_k,x_0,\dots,x_{k-1}\in \mathbb{Z}$, with $b_k\neq0$ and $x_0,x_1,\dots,\\x_{k-1}$ not all zeros. Suppose that the characteristic polynomial of $(x_n)_{n\geq0}$ is given by $$(x-a_1)(x-a_2)\dots (x-a_k),$$ where $a_i\in \mathbb{Z}$, $ a_i\neq a_j $ for $1\leq i\neq j\leq k$, and $\gcd(a_i, a_j)=1$ for all $i\neq j$. Let $p$ be a prime such that $p\nmid a_1a_2\cdots a_k$. If $x_0=0$, then the quotient set of $(x_n)_{n\geq0}$ is dense in $\mathbb{Q}_p$.
\end{thm}
\begin{ex}
	Suppose that $p_1,p_2$, and $p_3$ are distinct primes. Let $(x_n)_{n\geq 0}$ be a linear recurrence sequence defined by the recurrence relation $$x_{n}=(p_1+p_2+p_3)x_{n-1}-(p_1p_2+p_1p_3+p_2p_3)x_{n-2}+(p_1p_2p_3)x_{n-3}$$  for $n\geq 3$, where $x_0=0$, and $x_1$ and $x_2$ are any integers not both zero. The characteristic polynomial is equal to $(x-p_1)(x-p_2)(x-p_3)$. Hence, by Theorem \ref{new-thm1}, the quotient set of $(x_n)_{n\geq 0}$ is dense in $\mathbb{Q}_p$ for all primes $p\neq p_1,p_2, p_3$.
\end{ex}
In the following theorem, we consider $k$th order linear recurrence sequences whose characteristic polynomials have exactly two equal roots.  
\begin{thm}\label{thm2}
	Let $(x_n)_{n\geq0}$ be a linear recurrence of order $k\geq3$ satisfying
	\begin{align*}
	x_n=b_1x_{n-1}+b_2x_{n-2}+\dots+b_kx_{n-k},
	\end{align*}
	for all integers $n\geq k$, where $b_1,\dots,b_k,x_0,\dots,x_{k-1}\in \mathbb{Z}$, with $b_k\neq0$. Suppose that the characteristic polynomial of $(x_n)_{n\geq0}$ is given by $$(x-a_1)^2(x-a_2)(x-a_3)\dots (x-a_{k-1}),$$ 
	where $a_i\in \mathbb{Z}, a_i\neq a_j $ for $1\leq i\neq j\leq k-1$, and $x_0=x_1=\dots=x_{k-2}=0,x_{k-1}=1$. Let $p$ be a prime such that $p\nmid a_1a_2\cdots a_{k-1}$. If $a_i\not\equiv a_j\pmod{p}$ for all $i\neq j$, then the quotient set of $(x_n)_{n\geq 0}$ is dense in $\mathbb{Q}_p$.
\end{thm}
\begin{ex}
	Given an integer $a$, let $(x_n)_{n\geq 0}$ be a linear recurrence sequence defined by the recurrence relation 
	$$x_{n}=4ax_{n-1}-5a^2x_{n-2}+2a^3x_{n-3}$$ 
	for $n\geq 3$, where $x_0=x_1=0$ and $x_2=1$. The characteristic polynomial is equal to $(x-a)^2(x-2a)$.  By Theorem \ref{thm2},  the quotient set of $(x_n)_{n\geq 0}$ is dense in $\mathbb{Q}_p$ for all primes $p\nmid 2a$.
\end{ex}
\begin{thm}\label{thm3}
	Let $(x_n)_{n\geq0}$ be a linear recurrence of order $k\geq2$ satisfying
	\begin{align*}
	x_n=b_1x_{n-1}+b_2x_{n-2}+\dots+b_kx_{n-k},
	\end{align*}
	for all integers $n\geq k$, where $b_1,\dots,b_k,x_0,\dots,x_{k-1}\in \mathbb{Z}$, with $b_k\neq0$. Suppose that the characteristic polynomial of $(x_n)_{n\geq0}$ is given by $(x-a)^k$, where $a\in \mathbb{Z}$, and $x_0=x_1=\dots=x_{k-2}=0,x_{k-1}=1$. If $p$ is a prime such that $p\nmid a$, then the quotient set of $(x_n)_{n\geq 0}$ is dense in $\mathbb{Q}_p$.
\end{thm}	
\begin{rem}
	Let $a\in\mathbb{Z}$. Consider the $k$-th order linear recurrence sequence $(x_n)_{n\geq 0}$ generated by the recurrence relation 
	$$ x_{n}=\binom{k}{1}ax_{n-1}-\binom{k}{2}a^2x_{n-2}+\dots+(-1)^{k-1}\binom{k}{k}a^kx_{n-k}$$  for $n\geq k$, where $x_0=\dots=x_{k-2}=0,x_{k-1}=1$. Then, the quotient set of $(x_n)_{n\geq 0}$ is dense in $\mathbb{Q}_p$ for all primes $p$ not dividing $a$. This generalises \cite[Theorem 1.9]{reccurence} for the case $k=2$.\\
	Note that a linear recurrence sequence generated by a relation of the above form may not always have a dense quotient set in $\mathbb{Q}_p$. For example, consider the $p$-th order linear recurrence sequence $(x_n)$ generated by the recurrence relation 
	$$ x_{n}=\binom{p}{1}ax_{n-1}-\binom{p}{2}a^2x_{n-2}+\dots+(-1)^{p-1}\binom{p}{p}a^px_{n-p}$$ for $n\geq p$, where the initial values  $x_0,\dots,x_{p-1}\in\mathbb{Z}\backslash\{0\}$ have the same $p$-adic valuation. Then, the quotient set of $(x_n)$ is not dense in $\mathbb{Q}_p$ which follows from \cite[Theorem 1.10]{reccurence}.
\end{rem}
In case of third order recurrence sequences, we prove the following result where we do not need to fix all the initial values.
\begin{thm}\label{thm1}
		Let $(x_n)_{n\geq 0}$ be a third order linear recurrence sequence given by
		\begin{align*}
		x_n=b_1x_{n-1}+b_2x_{n-2}+b_3x_{n-3},
		\end{align*}
		for all integers $n\geq 3$, where $b_1,b_2, b_3,x_0,x_1, x_2\in \mathbb{Z}$, with $b_3\neq0$.
		Suppose that the characteristic polynomial of $(x_n)_{n\geq 0}$ is given by $(x-a)(x-b)(x-c)$, where $a,b,c\in \mathbb{Z}$. Let $p$ be a prime such that $p\nmid abc$. Then, the following hold.
		\begin{enumerate}
			\item[(a)] Suppose that $a=b=c$. If $p|x_0 $ and $p\nmid 4ax_1-x_2-3a^2x_0$, then the quotient set  of $(x_n)_{n\geq0}$ is dense in $\mathbb{Q}_p$. Moreover, if  $x_0=0,$ then the quotient set of $(x_n)_{n\geq 0}$ is dense in $\mathbb{Q}_p$ if and only if $4ax_1\neq x_2$.
			\item[(b)] Suppose that $a=c\neq b$. If $p|x_0$ and $p\nmid\left(a-b\right)\left(x_2-x_1(a+b)+x_0ab\right)$, then  the quotient set of $(x_n)_{n\geq 0}$ is dense in $\mathbb{Q}_p$.
	\end{enumerate}  
\end{thm}
	\begin{ex}
		Let $a\in\mathbb{Z}$ be such that $p\nmid a$, and let $(x_n)_{n\geq 0}$ be a linear recurrence sequence defined by the recurrence relation 
		$$x_{n+1}=3ax_n-3a^2x_{n-1}+a^3x_{n-2}$$ 
		for $n\geq 2$, where $x_0=0$, and $x_1$ and $x_2$ are any integers satisfying $\gcd(4a,x_2)=1$. Then, by Theorem \ref{thm1} (a),  the quotient set of $(x_n)_{n\geq 0}$ is dense in $\mathbb{Q}_p$.
		\end{ex}
	\section{Preliminaries}
	Let $r$ be a nonzero rational number. Given a prime number $p$, $r$ has a unique representation of the form $r= \pm p^k a/b$, where $k\in \mathbb{Z}, a, b \in \mathbb{N}$ and $\gcd(a,p)= \gcd(p,b)=\gcd(a,b)=1$. 
	The $p$-adic valuation of $r$ is defined as $\nu_p(r)=k$ and its $p$-adic absolute value is defined as $\|r\|_p=p^{-k}$. By convention, $\nu_p(0)=\infty$ and $\|0\|_p=0$. The $p$-adic metric on $\mathbb{Q}$ is $d(x,y)=\|x-y\|_p$. 
	The field $\mathbb{Q}_p$ of $p$-adic numbers is the completion of $\mathbb{Q}$ with respect to the $p$-adic metric. The $p$-adic absolute value can be extended to a finite normal extension $\mathbb{K}$ over $\mathbb{Q}_p$ of degree  $n$. 
	For $\alpha\in \mathbb{K}$, define $\|\alpha\|_p$ as the $n$-th root of the determinant of the matrix of linear transformation from the vector space $\mathbb{K}$ over $\mathbb{Q}_p$ to itself defined by $x\mapsto \alpha x$ for all $x\in \mathbb{K}$. 
	Also, $\nu_p(\alpha)$ is the unique rational number satisfying $\|\alpha\|_p=p^{-\nu_p(\alpha)}$.
	The ring of integers of $\mathbb{K}$, denoted by $\mathcal{O}$, is defined as the set of all elements in $\mathbb{K}$ with $p$-adic absolute value less than or equal to one. A function $f: \mathcal{O}\rightarrow \mathcal{O}$ is called analytic if there exists a sequence $(a_n)_{n\geq 0}$ in $\mathcal{O}$ such that 
	$$ f(z)=\sum_{n=0}^{\infty}a_nz^n$$ for all $z\in \mathcal{O}$.
	\par We recall  definitions of $p$-adic exponential and logarithmic function.  For $a\in\mathbb{K}$ and $r>0$, we denote $\mathcal{D}(a,r):=\{z\in \mathbb{K}: \|z-a\|_p<r\}$. Let $\rho=p^{-1/(p-1)}$.
	\par 
	For $z\in \mathcal{D}(0,\rho)$, the $p$-adic exponential function is defined as $$\exp_p(z)=\sum_{n=0}^{\infty}\frac{z^n}{n!}.$$ The derivative is given by $\exp_p'(z)=\exp_p(z)$.
	For $\mathcal{D}(1,\rho)$, the $p$-adic logarithmic function is defined as  $$\log_p(z)=\sum_{n=1}^{\infty}\frac{(-1)^{n-1}(z-1)^n}{n}.$$ 
	For $z\in \mathcal{D}(1,\rho)$, we have $\exp_p(\log_p(z))=z$. If $\mathbb{K}$ is unramified and $p\neq 2$, then $ \mathcal{D}(0,\rho)= \mathcal{D}(0,1)$ and  $\mathcal{D}(1,\rho)= \mathcal{D}(1,1)$. More properties of these functions can be found in \cite{gouvea}.
	\par Next, we state a result for analytic functions which will be used in the proofs of our theorems.
	\begin{thm}\cite[Hensel's lemma]{gouvea} Let $f: \mathcal{O}\rightarrow \mathcal{O}$ be analytic.
		Let $b_0\in \mathcal{O}$ be such that $\|f(b_0)\|_p<1$ and $\|f'(b_0)\|_p=1$. Then there exists a unique $b\in\mathcal{O}$ such that $f(b)=0$ and $\|b-b_0\|_p<\|f(b_0)\|_p$. 
	\end{thm}
Note that in \cite{gouvea}, Gouv\^{e}a states Hensel's lemma for polynomials with coefficients in $\mathcal{O}$. However, Hensel's lemma is also true and follows similarly for functions given by power series with coefficients in the ring $\mathcal{O}$.
We will only be considering $\mathbb{K}=\mathbb{Q}_p$ throughout this article. The following results are useful in proving denseness of quotient sets.
	\begin{thm}\label{zero-dense}\cite[Corollary 1.3]{piotr}
		Let $f\colon\mathbb{Z}_p\rightarrow \mathbb{Q}_p$ be an analytic function with a simple zero in $\mathbb{Z}_p$. Then, $R(f(\mathbb{N}))$ is dense in $\mathbb{Q}_p$.
	\end{thm}                              
	\begin{lem}\label{lem1}\cite[Lemma 2.1]{garciaetal}
		If $S$ is dense in $\mathbb{Q}_p$, then for each finite value of the $p$-adic valuation, there is an element of $S$ with that valuation.
	\end{lem}
	\section{Proof of the theorems}
	In the proofs, we will use certain representation of the $n$th term of linear recurrence sequence in terms of the roots of the characteristic polynomial. More details on such representations can be found in \cite{poonen}.
	\begin{proof}[Proof of Theorem \ref{new-thm1}]
		For $n\geq 0$, the $n$th term of the sequence $(x_n)$ is given by $$x_n=c_0a_1^n+c_1a_2^n+\dots+c_{k-1}a_k^n,$$
		where 
		\[C=\begin{bmatrix}
			c_0&c_1&\dots & c_{k-1}
		\end{bmatrix}^{t}\] 
		is given by $C=\frac{1}{\det(A)}\text{adj}(A)\cdot X_0$, where 
		\[ X_0= \begin{bmatrix}
				x_0\\
				x_1\\
				\vdots \\
				x_{k-1}				
			\end{bmatrix},
				A = \begin{bmatrix}
			1&1&\dots&1\\
			a_1&a_2&\dots&a_{k}\\
			a_1^2&a_2^2&\dots&a_k^2\\
			\vdots&\vdots&\ddots&\vdots\\
			a_1^{k-1}&a_2^{k-1}&\dots&a_k^{k-1}
		\end{bmatrix}.\]
		We define a  function $f$ as
		$$f(z):=\det(A)\left[c_0\exp_p{(z\log_p{a_1^{p-1}})}+\dots+c_{k-1}\exp_p{(z\log_p{a_k^{p-1}})}\right].$$ 
		Since $p\nmid a_1a_2\dots a_k$, $f$ is defined for all  $z\in \mathbb{Z}_p$ and $f(n)=\det(A)x_{n(p-1)}$ for all $n\in \mathbb{Z}_{\geq 0}$. Moreover, $\mathbb{Z}_{\geq 0}$ is dense in $\mathbb{Z}_p$. Therefore, $f$ is an analytic function from $\mathbb{Z}_p$ to $\mathbb{Z}_p$.
		We have, 
		$$f(0)=\det(A)(c_0+c_1+\dots+c_{k-1})=\det(A)x_0=0$$ and 
		$$f'(0)=\det(A)(c_0\log_p{a_1^{p-1}}+c_1\log_p{a_2^{p-1}}+\dots+c_{k-1}\log_p{a_{k}^{p-1}}).$$
		Suppose that $f'(0)=0$. Since $\gcd(a_i,a_j)=1$ for all $i\neq j$, therefore, $a_1^{p-1},\dots,a_k^{p-1}$ are multiplicatively independent i.e, $(a_1^{p-1})^{u_1}(a_2^{p-1})^{u_2}\dots(a_k^{p-1})^{u_k}=1$ for some integers $u_1,u_2,\dots,u_k$ only if $u_1=u_2=\dots=u_k=0$. Hence, $$\log_p{a_1^{p-1}},\log_p{a_2^{p-1}},\dots,\log_p{a_{k}^{p-1}}$$ are linearly independent over $\mathbb{Z}$. Thus, if $f'(0)=0$ then $c_0=c_1=\dots=c_{k-1}=0$ which is not possible. Hence, $f'(0)$ is nonzero. Therefore, $0$ is a simple zero of $f$ in $\mathbb{Z}_p$. By Theorem \ref{zero-dense}, $R(f(\mathbb{N}))=R((x_{n(p-1)}))$ is dense in $\mathbb{Q}_p$.  Hence, the quotient set of $(x_n)_{n\geq 0}$ is dense in $\mathbb{Q}_p$. 
		\end{proof}
		\begin{proof}[Proof of Theorem \ref{thm2}]
		The $n$th term of the sequence is given by 
		\begin{align*}
			x_n&=a_1^n(c_0+c_1n)+c_2a_2^n+c_3a_3^n+\dots+c_{k-1}a_{k-1}^n\\
			&=a_1^n(c_0+c_1n+c_2(a_2a_1^{-1})^n+c_3(a_3a_1^{-1})^n+\dots+c_{k-1}(a_{k-1}a_1^{-1})^n),
		\end{align*}
		where 
		\[C=\begin{bmatrix}
			c_0&c_1&\dots &c_{k-1}
		\end{bmatrix}^{t}\] 
		is given by $C=\frac{1}{\det(A)}\text{adj}(A)\cdot X_0$, where 
		\[ X_0=  \begin{bmatrix}
			0\\
			0\\
			\vdots \\
			0\\
			1
		\end{bmatrix},
		A = \begin{bmatrix}
			1&0&1&\dots&1\\
			a_1 & a_1 & a_2 & \dots & a_{k-1}\\
			a_1^2 & 2a_1^2 & a_2^2 & \dots & a_{k-1}^2\\
			\vdots& \vdots & \vdots & \ddots & \vdots\\
			a_1^{k-1} & (k-1)a_1^{k-1} & a_2^{k-1} & \dots & a_{k-1}^{k-1}
		\end{bmatrix}.\]
		\par 
		We define an analytic function $f:\mathbb{Z}_p\rightarrow\mathbb{Z}_p$ as
		\begin{align*} f(z)&:=\det(A)\exp_p{(z\log_p{(a_1)^{p-1}})}(c_0+c_1z(p-1)+c_2\exp_p{(z\log_p{(a_2a_1^{-1})^{p-1}})}\\
			&\hspace{1.2cm}+\dots+ c_{k-1}\exp_p{(z\log_p{(a_{k-1}a_1^{-1})^{p-1}})}).
		\end{align*}
		Then,  $f(n)=\det(A)x_{n(p-1)}$ for all $n\in \mathbb{Z}_{\geq 0}$.
		Also, we have
		\begin{align*}
			f(0)=\det(A)(c_0+c_2+\dots+c_{k-1})=\det(A)x_0=0
		\end{align*}
		and 
		\begin{align*}
			f'(0)&=\det(A)(c_1(p-1)+c_2\log_p{(a_2a_1^{-1})^{p-1}}+\dots+c_{k-1}\log_p{(a_{k-1}a_1^{-1})^{p-1}}\\
			&\hspace{1.2cm} +(c_0+c_2+\dots+c_{k-1})\log_p{(a_1)^{p-1}})\\
			&=\det(A)(c_1(p-1)+c_2\log_p{(a_2a_1^{-1})^{p-1}}+\dots+c_{k-1}\log_p{(a_{k-1}a_1^{-1})^{p-1}}).
			\end{align*}
		We find that $\det(A)c_1=(-1)^{k+1}\prod_{1\leq i<j\leq (k-1)}(a_i-a_j)$. By the hypothesis, we have $p\nmid \det(A)c_1$. Using the definition of $\log_p(z)$, we obtain that $p$ divides $\log_p{(a_ia_1^{-1})^{p-1}}$ for $2\leq i \leq k-1$. Therefore, $p\nmid f'(0)$ which implies $f'(0)$ is nonzero. Hence, $0$ is a simple zero of $f$ in $\mathbb{Z}_p$. By Theorem \ref{zero-dense}, $R(f(\mathbb{N}))=R(x_{n(p-1)})$ is dense in $\mathbb{Q}_p$. Hence, the quotient set of $(x_n)_{n\geq 0}$ is dense in $\mathbb{Q}_p$.
	\end{proof}
	\begin{proof}[Proof of Theorem \ref{thm3}]
		The $n$th term of the sequence is given by
		$$x_n=a^n(c_0+c_1n+\dots +c_{k-1}n^{k-1}),$$ where 
		$$C=\begin{bmatrix}
			c_0&c_1&\dots &c_{k-1}
		\end{bmatrix}^{t}$$ is given by $C=\frac{1}{\det(A)}\text{adj}(A)\cdot X_0$, where
		\[ X_0=  \begin{bmatrix}
			0\\
			0\\
			\vdots \\
			0\\
			1
		\end{bmatrix},
		A = \begin{bmatrix}
			1 & 0 & 0 & \dots & 0\\
			a & a & a & \dots & a\\
			a^2 & 2a^2 & 2^2a^2 & \dots & 2^{k-1}a^2\\
			\vdots& \vdots & \vdots & \ddots & \vdots\\
			a^{k-1} & (k-1)a^{k-1} & (k-1)^2a^{k-1} & \dots & (k-1)^{k-1}a^{k-1}
		\end{bmatrix}.\]
		We simplify $C=\frac{1}{\det(A)}\text{adj}(A)\cdot X_0$ and obtain
		\begin{align*}
			\begin{bmatrix}
				c_0 \\ c_1 \\ 
				c_2\\\vdots \\ c_{k-1}
			\end{bmatrix}=\begin{bmatrix}
				1 & 0 & \dots & 0\\
				1 & 1 & \dots & 1\\
				1 & 2 &\dots & 2^{k-1}\\
				\vdots& \vdots & \ddots & \vdots\\
				1 & (k-1)  & \dots & (k-1)^{k-1}
			\end{bmatrix}^{-1}\begin{bmatrix}
				0 \\ 0 \\ 0 \\ \vdots \\ 1/a^{k-1}
			\end{bmatrix}.
		\end{align*}  
	
		\par 
		Next, we consider an analytic function $f:\mathbb{Z}_p\rightarrow \mathbb{Z}_p$ defined as
		\begin{align*} f(z)&:=\det(A)\exp_p{(z\log_p{(a^{p-1})})}(c_0+c_1(p-1)z+c_2(p-1)^2z^2+\\
		&\hspace{2.6cm}+\dots+c_{k-1}(p-1)^{k-1}z^{k-1}).
		\end{align*}
		Let
		\begin{align*}
		h(z):=\exp_p{(z\log_p{(a^{p-1})})}
		\end{align*}  
		and 
		\begin{align*}
		g(z):=\det(A)(c_0+c_1(p-1)z+c_2(p-1)^2z^2+\dots+c_{k-1}(p-1)^{k-1}z^{k-1}).
		\end{align*}
		We have $\|a^n\|_p=1$ and $h(n)=a^{n(p-1)}$ for all positive integers $n$. Hence, $\|h(z)\|_p=1$ for all $z\in \mathbb{Z}_p$. Therefore, $f(z)=0$ if and only if $g(z)=0$ for some $z\in\mathbb{Z}_p$. 
		We have, $g(0)=\det(A)c_0=\det(A)x_0=0$ and $g'(0)=\det(A)c_1(p-1)$. Using Lemma 2.2 of \cite{vandermonde}, we find that 
		\begin{align*}
		c_1=\frac{(-1)^{k}}{a^{k-1}(k-1)}.
		\end{align*}
		Thus, $c_1\neq 0$ for all $k\geq 2$. Therefore, $0$ is a simple zero of $f$ in $\mathbb{Z}_p$. By Theorem \ref{zero-dense}, $R(f(\mathbb{N}))=R(x_{n(p-1)})$ is dense in $\mathbb{Q}_p$, which yields that the quotient set of $(x_n)_{n\geq 0}$ is dense in $\mathbb{Q}_p$. 
	\end{proof}
	\begin{proof}[Proof of Theorem \ref{thm1}]
	We first prove part (a) of the theorem.
	For $n\geq 0$, the $n$th term of the sequence is given by the formula
	$$x_n=a^n(c_0+c_1n+c_2n^2),$$ where
	\begin{align*}
	 c_0&=x_0,\\
	 c_1&=\frac{4ax_1-x_2-3a^2x_0}{2a^2},\\
	 c_2&=\frac{x_2-2ax_1+a^2x_0}{2a^2}.
	 \end{align*}
We define a function $f$ as 
	  $$f(z):= 2a^2\exp_p(z\log_p{a^{p-1}})(c_0+c_1(p-1)z+c_2(p-1)^2z^2).$$
	  	Since $p\nmid a$, $f$ is defined for all  $z\in \mathbb{Z}_p$ and $f(n)=2a^2x_{n(p-1)}$ for all $n\in \mathbb{Z}_{\geq 0}$. Moreover, $\mathbb{Z}_{\geq 0}$ is dense in $\mathbb{Z}_p$. Therefore, $f$ is an analytic function from $\mathbb{Z}_p$ to $\mathbb{Z}_p$.
	 	  We have, $$f(0)=2a^2c_0=2a^2x_0\equiv 0\pmod{p}$$ and
	   $$f'(0)\equiv 2a^2c_1(p-1)=2a^2(4ax_1-x_2-3a^2x_0)(p-1)\not\equiv 0 \pmod{p}.$$ 
	   Therefore, by Hensel's lemma, $f$ has a zero $z_0$ in $\mathbb{Z}_p$ such that $z_0\equiv 0\pmod{p}$. Since $f$ has a power series expansion with $p$-adic integral coefficients, we have $f'(z_0)\equiv f'(0)\pmod{p}$. Hence, $z_0$ is a simple zero of $f$ in $\mathbb{Z}_p$. Therefore, by Theorem \ref{zero-dense}, $R(f(\mathbb{N}))=R((x_{n(p-1)}))$ is dense in $\mathbb{Q}_p$. Hence, the quotient set of $(x_n)_{n\geq 0}$ is dense in $\mathbb{Q}_p$.
	   \par 
	  Next, if $x_0=0$, then $c_0=0$ and $c_1=\frac{4ax_1-x_2}{2a^2}$. We have $f(0)=0$. Suppose that $4ax_1\neq x_2$. Then, $f'(0)\neq 0$ which implies that $0$ is a simple zero of $f$. Therefore, by Theorem \ref{zero-dense}, the quotient set of $(x_n)_{n\geq 0}$ is dense in $\mathbb{Q}_p$. Conversely, suppose that $4ax_1=x_2$. This gives $c_1=0$, and hence $x_n=a^nc_2n^2$. If $c_2= 0$, then $x_n=0$ for all $n$. If $c_2\neq 0$, then the quotient set of $(x_n)_{n\geq 0}$ is equal to the quotient set of $\{a^nn^2:n\in\mathbb{Z}_{\geq 0}\}$. Since $\nu_p(a^nn^2)=2\nu_p(n)$, the $p$-adic valuation of these elements is even for all $n\in\mathbb{Z}_{>0}$. Therefore, by Lemma \ref{lem1}, the quotient set of $(x_n)_{n\geq 0}$ is not dense in $\mathbb{Q}_p$. This completes the proof of part (a) of the theorem.
	 \par 
	  Next, we prove part (b) of the theorem. For $n\geq 0$, the $n$th term of the sequence is given by 
	   $$x_n=c_0a^n+c_1na^n+c_2b^n=a^n(c_0+c_1n+c_2(ba^{-1})^n),$$ where
	   \begin{align*}
	    c_0&=\frac{b^2x_0-2abx_0-x_2+2ax_1}{(b-a)^2}, \\
	    c_1&=\frac{x_2-x_1(a+b)+x_0ab}{a(a-b)}, \\
	    c_2&=\frac{x_2-2ax_1+a^2x_0}{(b-a)^2}.
	    \end{align*}
	   Since $p\nmid ab(a-b)$, we can define an analytic function $f:\mathbb{Z}_p\rightarrow\mathbb{Z}_p$ as 
	   \begin{align*}
	   f(z):=\exp_p(z\log_p{a^{p-1}})(c_0+c_1z(p-1)+c_2\exp_p(z\log_p({ba^{-1})^{p-1}})),
	   \end{align*}
       which satisfies the equation  $f(n)=x_{n(p-1)}$ for all $n\geq 0$. \\
       Now, we have
       \begin{align*}
       f(0)=c_0+c_2=x_0\equiv 0\pmod{p}
       \end{align*}
       and
       \begin{align*}
       f'(0)=c_1(p-1)+c_2\log_p(ba^{-1})^{p-1}+(c_0+c_2)\log_p{a^{p-1}}\not\equiv 0\pmod{p}.
       \end{align*}
       Therefore, by Hensel's lemma, $f$ has a zero $z_0$ in $\mathbb{Z}_p$ such that $z_0\equiv 0\pmod{p}$. Since $f$ has a power series expansion with $p$-adic integral coefficients, we have $f'(z_0)\equiv f'(0)\pmod{p}$. Hence, $z_0$ is a simple zero of $f$ in $\mathbb{Z}_p$. Therefore, by Theorem \ref{zero-dense}, $R(f(\mathbb{N}))=R(x_{n(p-1)})$ is dense in $\mathbb{Q}_p$. Hence, the quotient set of $(x_n)_{n\geq 0}$ is dense in $\mathbb{Q}_p$.
\end{proof}


\end{document}